\documentclass[a4paper,10pt]{article}
\usepackage[utf8]{inputenc}
\usepackage{amsmath,amssymb,amsthm}

\newtheorem{lemma}{Lemma}
\newtheorem{theorem}{Theorem}
\newtheorem{proposition}{Proposition}
\newcommand{\R}{\mathbb{R}}
\newcommand{\sign}{\text{\rm sign}}
\newcommand{\Breg}[1]{B_{#1}}
\newcommand{\BregS}{B_S}
\newcommand{\dual}[1]{\left\langle #1 \right\rangle }
\newcommand{\f}{{f_\theta}} 
\newcommand{\g}{{g_\theta}} 
\newcommand{\rs}{r_*}

\title{A note on the approximate symmetry of Bregman distances }
\author{Stefan Kindermann\footnote{Industrial Mathematics Institute, Johannes Kepler University Linz, 
Alternbergergstra{\ss}e~69, 4040 Linz (kindermann@indmath.uni-linz.ac.at).}}
\date{}

\begin{document}

\maketitle
\begin{abstract}
The Bregman distance $\Breg{\xi_x}(y,x)$, $\xi_x \in \partial J(y),$ associated 
to a convex sub-differentiable functional $J$ is known to be in general non-symmetric in 
its arguments $x$, $y$. In this note we address the question when  
Bregman distances can be bounded against each other 
when the arguments are switched, i.e., 
if some constant $C>0$ exists such that for all $x,y$ on a convex set $M$ it holds that  
$\frac{1}{C} \Breg{\xi_x}(y,x) \leq \Breg{\xi_y}(x,y) \leq C \Breg{\xi_x}(y,x).$
We state sufficient conditions for such an inequality and prove in particular that
it holds for the $p$-powers of the $\ell_p$ and $L^p$-norms when $1 < p <\infty$. 
\end{abstract} 

\section{Introduction}
For a convex sub-differentiable functional $J$ on a convex subset $M$ of 
a Banach space $X$, the Bregman distance between 
$x,y \in M$ with a chosen element $\xi_x \in \partial J(x) $ is 
defined by 
\[ \Breg{\xi_x}(y,x) = J(y) - J(x) - \dual{\xi_x,y-x} \qquad \xi_x \in \partial J(x). \] 
It is a useful tool in the analysis of optimization problems, and in particular, in the 
Banach space theory of variational regularization, it has become a useful tool to measure 
errors; see, e.g., 
\cite{RS06,HKPS07}. Note that for Hilbert spaces $X$ with \mbox{$J = \frac{1}{2}\|.\|^2$}, the Bregman distance 
equals  $\Breg{\xi_x}(y,x) = \frac{1}{2}\|x-y\|^2$, hence convergence in Bregman distance is often used to 
generalized results on  norm convergence in  Hilbert spaces. 

However, in general, in contrast to the Hilbert space case, the Bregman distance is not symmetric 
in its arguments, i.e.,  $\Breg{\xi_x}(y,x) \not = \Breg{\xi_y}(x,y)$.   It has been observed that 
the conditions for convergence rates for Tikhonov regularization are not the same when 
convergence is measured in the Bregman distance and in its switched version \cite{K16,F18}. 
It is hence of interest, to study the questions when switching the arguments do not 
change the topology of Bregman convergence, i.e., when the switched Bregman 
distance can be bounded by a constant times the original one. 

More precisely, we investigate conditions, when there exists a constant $C_M$ such that 
for all $x,y \in M$ it holds that 
\begin{equation}\label{main} 
\begin{split}
\frac{1}{C_M} \Breg{\xi_x}(y,x) \leq \Breg{\xi_y}(x,y) &\leq C_M \Breg{\xi_x}(y,x) 
 \qquad   \forall \xi_x \in \partial J(x), \xi_y \in \partial J(y),
\end{split}
\end{equation} 
where $M$ is some convex set in $X$.  If \eqref{main} holds, then, when investigating the 
Bregman convergence  of an approximating sequence $y = x_n$ to $x$, 
it does not matter 
which of the two variants,  i.e., $\Breg{\xi_x}(x_n,x)$ 
or $\Breg{\xi_{x_n}}(x,x_n)$ 
one considers.

It is trivial that \eqref{main} holds with the constant $C=1$ 
for $J = \frac{1}{2}\|.\|_H^2$ with $\|.\|_H$ a Hilbert space norm.  However, the inequality does not hold 
in general, as some simple counterexample show. 

A practically useful result that we  show in this paper  is that 
for $p$-powers of the $\ell^p$ or $L^p$-norms we can 
establish the inequality \eqref{main}. The corresponding results reads as follows: 

\begin{theorem}\label{mainth}
For $1 < p <\infty$, 
let $X$ be the space of $p$-summable sequence $X = \ell_p$ or 
 $p$-integrable functions $X = L^p(\Omega)$. Let $J$ be the $p$-power of the corresponding 
norms $J = \|.\|^p$. Then there exist a constant $C_p$ such that 
\begin{equation}\label{main_p} 
\frac{1}{C_p} \Breg{\xi_x}(y,x) \leq \Breg{\xi_y}(x,y) \leq C_p \Breg{\xi_x}(y,x) 
\end{equation} 
for all $x,y \in X$ and $\forall \xi_x \in \partial J(x), \xi_y \in \partial J(y)$. 
A constant $C_p$  is given by $C_p = 2 \max\{\frac{1}{p-1}, p-1\}$. 
\end{theorem}

\section{Sufficient conditions for approximate symmetry}
Before we investigate some sufficient conditions for \eqref{main}, we illustrate the 
problem by some (simple) examples.

Consider the abs-functional on $\R$: $J(x) = |x|$. Its subgradient is $\partial J(x)=\sign(x)$ for $x\not =0$ and 
multi-valued $\partial J(x) \in [-1,1]$ at $x = 0$. Hence, for $x\not = 0$,
\[ B_{\xi_0}(x,0)  = |x| - [-1,1]x = |x| [0,2] \qquad  B_{\xi_x}(0,x) = -|x| + \sign(x)x = 0. \]
Thus, \eqref{main} cannot hold in this case.  The fact that the subgradient is here multi-valued 
at one of the arguments 
is not responsible for the violation of inequality ~\eqref{main}  as the example $x >0$, $y = -\epsilon$, $\epsilon >0$
shows:
\[B_{\xi_\epsilon}(x,\epsilon) = 2 x \qquad B_{\xi_x}(\epsilon,x) = 2 \epsilon. \] 
Since $\epsilon$ can be chosen arbitrary small, no constant for  \eqref{main} exists. 

Furthermore, let us illustrate  that the lack of differentiability is not a sole reason for a violation of 
 \eqref{main}. Indeed, we might introduce a Huber-type smoothing of the previous example:
 \[ J(x) = \begin{cases} |x| & |x| \geq 1 \\
                                      x^2 & |x| <1 \end{cases}, \]
such that  $J(x) \in C^1(\R)$. 
Then, for $x >1$ and $y = 1-\epsilon$ with $1>\epsilon>0$, we have  
\[B_{\xi_y}(x,y) = x- 2 x y + y^2 = (1-\epsilon)^2 - x(1 - 2 \epsilon) = 
1 - x + O(\epsilon) \] 
while 
\[B_{\xi_x}(y,x) = (1-\epsilon) \epsilon = O(\epsilon). \] 
Thus,  \eqref{main} cannot hold uniformly in  $\epsilon$. 
Note that  similar counter-examples  to  \eqref{main} can be constructed where $J$ is in $C^\infty$.

The main inequality \eqref{main} can certainly be established if appropriate upper and lower 
bounds for the Bregman distances can be verified. For functionals involving powers of norms, 
Xu and Roach \cite{XR91} established useful estimates  that relate upper and lower 
bounds for  Bregman distances to smoothness and strict convexity. 
 It is thus not surprising 
that $C_M$ will be related to the ratio of quantities representing smoothness and strict convexity. 
Before we elaborate on that, let us state  that Theorem~\ref{mainth} cannot be 
obtained by a simple application of 
the Xu-Roach inequalities but requires a more detailed analysis. 
Indeed, for the $l^p$- or $L^p$-case the Xu-Roach inequalities imply the following estimates 
on bounded sets:
\[ B_{\xi_x}(y,x) \leq C \|x-y\|^{\min\{p,2\}} \qquad 
B_{\xi_x}(y,x) \geq C \|x-y\|^{\max\{p,2\}}.  \] 
Thus, except for  the trivial Hilbert space case $p = 2$, the exponents do not match 
to establish  \eqref{main} in a simple manner. 

Before we state a general result, 
we give some reformulations of the main inequality:
Obviously, for \eqref{main} to hold, it is enough that 
\begin{equation}\label{main1}
 \Breg{\xi_y}(x,y)  \leq  C_M \Breg{\xi_x}(y,x)  
\end{equation}  
for all $x,y \in M$ and all $\xi_y \in \partial J(y), \xi_x \in \partial J(x)$  
as the other inequality follows easily from that one by switching arguments. 

We may also  introduce the symmetric Bregman distance with $\xi_x \in \partial J(x)$ and 
$\xi_y \in \partial J(y)$ 
\[ \BregS(x,y) = \Breg{\xi_y}(x,y)  + \Breg{\xi_x}(y,x) = \dual{\xi_x- \xi_y ,x-y}.  \] 
Then the main inequality can be expressed in terms of the symmetric version. 
Indeed, by adding $C_M   \Breg{\xi_y}(x,y)$ to  \eqref{main1}, then we get 
the following lemma. 
\begin{lemma}\label{lem1} 
Inequality \eqref{main1} holds with some constant $C_M$  
if and only if 
\begin{equation}\label{main2} 
\begin{split}
\Breg{\xi_y}(x,y)  &\leq   \eta_M  \BregS (x,y) 
\end{split} 
\end{equation} 
holds with some constant $\eta_M <1$. 
The constants are related by 
\[ C_M = \frac{\eta_M}{1-\eta_M} \qquad \eta_M = \frac{C_M}{C_M+1}. \]
\end{lemma} 
Note  that \eqref{main2} always trivially holds with $\eta = 1$.

We also point out that the approximate symmetry holds for a functional $J$ if it holds for its dual $J^*$. 
Indeed, from the  well-know relations \cite{ET99}
\[ \dual{\xi_y,y} = J(y) + J^*(\xi_y),  \qquad 
x^* \in  \partial J(x) \Leftrightarrow  x \in  \partial J^*(x^*),  \] 
we can conclude that 
\begin{equation}\label{dualf} \Breg{\xi_y}(x,y) = \Breg{x}^*(\xi_y,\xi_x), \end{equation}
where  $\Breg{x}^*$ is the Bregman distance associated to the dual functional $J^*$ and 
$\xi_y,$ $\xi_x$ are elements of the subgradients at $y$ and $x$, respectively. 
Thus, we have 
\begin{lemma}
The inequality \eqref{main} holds for the Bregman distance of a  functional $J$  on $M$ with constant $C_M$ if and
only if it holds with the same constant for the  Bregman distance of the  dual  functional $J^*$ on 
$\partial J(M).$
\end{lemma}

Let us now state a  general result on approximate symmetry in  the following theorem:
\begin{theorem}\label{main3}
Suppose that $\partial J$ is strongly monotone on $M$
\[ \dual{\partial J(x) -\partial J(y),x-y} \geq c_0 \|x-y\|^2 \qquad \forall x,y \in M, \]
and 
Lipschitz continuous on $M$ with Lipschitz constant $L$.  Then
\eqref{main1} holds on for all $x,y \in M$
with a constant $C_M = \frac{L}{c_0}$.
\end{theorem} 
\begin{proof}
The result follows from the estimates 
\begin{align*}   \Breg{\xi_x}(y,x) &= \int_0^1 \dual{\partial J(x + t(y-x))- \partial J(x),y-x} dt  \\
& = 
L \int_0^1  t dt  \|x-y\|^2 = \frac{L}{2} \|x-y\|^2 \qquad \mbox{ and } \\ 
  \Breg{\xi_y}(x,y) &= \int_0^1 \dual{\partial J(y + t(x-y))- \partial J(y),x-y} dt   \\
  & =  
\int_0^1 \frac{1}{t} \dual{\partial J(y + t(x-y))(x-y)- \partial J(y),(y + t(x-y)-y} dt \\
&\geq c_0 \int_0^1   \frac{1}{t} \|y + t(x-y)-y\|^2 dt = \frac{c_0}{2}\|x-y\|^2 .
\end{align*}
\end{proof}

Note that for a $C^2(M)$-functional with strongly monotone gradient, 
we can estimate the constants in the previous theorem by 
upper and lower bounds for  the second derivative:
\begin{equation}\label{sder}  C_M \leq \frac{ \sup_{x \in M, \|h\| \leq 1}  |\dual{\nabla^2J(x) h,h}|}
{ \inf_{x \in M, \|h\| \leq 1}  |\dual{\nabla^2J(x) h,h}| }.
\end{equation}

As an example, we consider the square-root regularization of the abs-functional, which is often employed when 
a differentiable approximation to the $\ell^1$-norm is needed: 
\[ J: \R \to \R, \quad  x \to \sqrt{x^2 +\epsilon}, \qquad \epsilon >0. \] 
We show that the associated Bregman distance satisfies \eqref{main1} on bounded sets: 
We have for the derivatives 
\[ J'(x) = \frac{x}{\sqrt{x^2+\epsilon}} \qquad 
J''(x) = \frac{\epsilon}{(x^2+\epsilon)^\frac{3}{2}}. \]  
A direct calculation reveals a constant 
\[ C_M = \sup_{x,y\in M} \frac{\sqrt{y^2+\epsilon}}{\sqrt{x^2+\epsilon}}. \]
In particular, if  we restrict the functional to bounded sets, $|x|,|y| \leq R$ then 
\[ C_M  \leq \sqrt{1 + \tfrac{R^2}{\epsilon}} \leq 1 + \tfrac{R}{\sqrt{\epsilon}}. \] 
Using  \eqref{sder} 
gives a slightly worse bound 
 \[ C_M \leq \frac{\sup_{|x| \leq R} J''(x)}{\inf_{|x| \leq R} J''(x)}   = 
  \frac{\sup_{|x| \leq R} (|x|^2 +\epsilon)^\frac{3}{2}}{\inf_{|x|  \leq R}   |x|^2 +\epsilon)^\frac{3}{2}} 
  \leq \left( \tfrac{R^2}{\epsilon} + 1\right)^\frac{3}{2} \leq 
  1 +  O(\tfrac{R}{\sqrt{\epsilon}})^\frac{3}{2} \] 
  for $\tfrac{R}{\sqrt{\epsilon}} >>1$.

The same theorem can be used to prove \eqref{main} for the $p$-power functional in 
Theorem~\eqref{mainth}
if we restrict $x,y$ to bounded sets.  Below, however, we obtain the result without the restriction 
to bounded sets by a direct calculation.

For the sake of generalizations, we slightly improve Theorem~\ref{main3} by 
showing that the conditions in \eqref{main3} can 
be localized in the following sense:
\begin{proposition}\label{Propmain3}
Suppose that \eqref{main} holds for all $x,y$ in  \mbox{$M_0 = \{\|x-y\| \leq \epsilon\}$}. 
Then  there exists a constant  $C_M$ such that 
 \eqref{main} holds in  $M = \{\|x-y\| \leq R\}$, $R>\epsilon$, with  
 \[ C_M = \frac{R-\epsilon}{\epsilon} + \frac{R C_{M_0}}{\epsilon}. \]
\end{proposition} 
\begin{proof}
Denote by  $B_R(x) := \{x + z\,|\,\|z\| \leq R\}$  a ball around $x$ with radius $R$, and 
fix $x \in X$. According to Lemma~\ref{lem1} and the hypothesis, we have a 
constant $\eta_{M_0} = \frac{C_{M_0}}{1+C_{M_0}}$ such that 
\eqref{main2} holds for all $z$ in $B_\epsilon(x)$. 
Take $y \in B_R(x)$ and set $z = x + \lambda (y-x)$ and $\lambda =\frac{\epsilon}{R}<1$. 
Note that 
\[ z-x = \lambda (y-x), \quad y-z = (1-\lambda) (y-x). \]
Then $z \in B_\epsilon(x)$. 
We have 
\begin{align*} 
\Breg{\xi_x}(y,x) &= J(y) - J(x) - \dual{\xi_x,y-x} \\ 
& = J(z) - J(x)  -\dual{\xi_x,z-x} + J(y) - J(z) - \dual{\xi_x,y-z}  \\
&\leq \eta_{M_0}\dual{\xi_z-\xi_x,z-x} \\
& \qquad  +  \left(J(y) - J(z) - \dual{\xi_z,y-z} \right) + \dual{\xi_z -\xi_x,y-z} \\
& \leq \eta_{M_0}\dual{\xi_z-\xi_x,z-x}  + 
\dual{\xi_y -\xi_z, y-z} +  \dual{\xi_z -\xi_x,y-z} \\
& = 
\left( {\eta_{M_0}}\lambda + (1 - {\lambda})\right) \dual{\xi_z -\xi_x,y-x}
+ (1 -{\lambda}) \dual{\xi_y -\xi_z, y-x}  \\
& \leq  \max\left\{  {\eta_{M_0}}\lambda + 1 - {\lambda}, 
 1 -{\lambda} \right\} \left(  \dual{\xi_z -\xi_x,y-x} + \dual{\xi_y -\xi_z, y-x} \right) \\ 
& =   \left(\frac{\eta_{M_0}}{\lambda} + 1 - {\lambda}\right)
\dual{\xi_y -\xi_x, y-x} =
\left({\eta_{M_0}}{\lambda} + 1 - {\lambda}\right) \Breg{S}(y,x),
 \end{align*}
which proves  \eqref{main2} with $\eta_M = 
\left({\eta_{M_0}}{\lambda} + 1 - {\lambda}\right)$   
for all $y \in B_R(x)$, from which the constant $C_M$ can be calculated. 
Note that  the estimate in the penultimate line is valid because both expression 
 with  brackets $\dual{.,}$ in the previous line are nonnegative.  
Switching the role of $x,y$ proves the assertion.
\end{proof}

\section{The proof of Theorem~\ref{mainth}}
We prove Theorem~\ref{mainth} by a direct calculation of the constant $C_M$. 
Slightly more general as stated in the theorem, 
 we consider the $p$-power of a Hilbert space norm. For  $p \in (1,\infty)$, define  
\[ J: H \to \R \quad x \to \frac{1}{p} \|x\|^p\] 
with $\|.\|$ the norm on $H$. 
Note that $J$ is continuously differentiable with 
\[ \partial J(x) =  e_x   \|x\|^{p-1} \qquad   e_x :=\begin{cases} \frac{x}{\|x\|} &  x \not = 0, \\ 
                                                   0 &  x = 0. \end{cases}  \] 
To prove the theorem, we establish the following technical lemma: 

\begin{lemma}\label{lemmatech}
Let $1 < p <2$, and define the functions 
\begin{alignat*}{2} 
 \f(r) &= \frac{1}{p} r^p +  (1- \frac{1}{p})  - r \theta, \quad &r \geq 0, \theta \in [-1,1],  \\
 \g(r) &=
\frac{1}{p}  + (1  -\frac{1}{p}) r^p  - r^{p-1} \theta   \quad &r \geq 0, \theta \in [-1,1] . 
\end{alignat*}                                                   
Then 
\begin{equation}
\max_{r\geq 0,  \theta \in [-1,1] }  \frac{\f(r)}{\g(r)} \leq 
\frac{1}{p-1} \max_{r\geq1} 
\frac{{r}^{p-1} + 1}{{r}^{p-1} +  {r}^{p-2}} \leq \frac{2}{p-1}.
\end{equation}
\end{lemma}
\begin{proof}
Both functions are positive with a zero only at $r =1$ and $\theta = 1$.                             
The ratio function $\frac{\f(r)}{\g(r)}$ is thus a positive smooth function except for $r = 1,\theta = 1$ and 
has  the following limit values for any $\theta \in [-1,1]$  
\[ \lim_{r\to 0} \frac{\f(r)}{\g(r)} = p-1,
\qquad  \lim_{r\to \infty} \frac {\f(r)}{\g(r)}  = \frac{1}{p-1}, \qquad
\lim_{r\to 1} \frac{\f(r)}{\g(r)}= 1. \]                                                   
In the case $\theta=1$, the last limit can be  evaluated by two-times applying de~l'Hospital's rule. 
Thus, for fixed $\theta$ the ratio function can be continuously extended to all $r\geq0$.

For a fixed value $r \not =1$, we calculate the derivative 
\begin{align*} 
\frac{\partial}{\partial \theta} \frac{\f(r)}{\g(r)} &= 
\frac{1}{\g(r)^2}\left(  \f(r) r^{p-1} -  r \g(r)  \right) \\
& = \frac{1}{\g(r)^2}\left(\frac{1}{p} r^{2p-1} +  (1- \frac{1}{p}) r^{p-1} 
- \frac{1}{p}r   + (1  -\frac{1}{p}) r^{p+1} \right)
\end{align*}
and observed that this value  is always of the same  sign for a fixed  $r$. Thus 
the maximal value with respect to $\theta$ is attained at the boundary of $[-1,1]$:
\[ \max_{\theta \in [-1,1]}  \frac{\f(r)}{\g(r)} \leq 
 \max_{\theta \in \{-1,1\}}  \frac{\f(r)}{\g(r)}. \]

We now investigate those interior maxima of the ratio function with respect to $r$ 
in the interval $(0,1)$ and $(1,\infty)$ that 
 have a value larger than the largest currently found $\frac{\f(\infty)}{\g(\infty)}=\frac{1}{p-1}$. 
As the ratio  is a smooth function, we may conclude from the optimality condition of first  order that 
at such a maximum $\rs$, we have 
\[  \f'(\rs) = \frac{\f(\rs)}{\g(\rs)} \g'(\rs). \]
Setting $\lambda =  \frac{\f(\rs)}{\g(\rs)}>\frac{1}{p-1} $ (recall that only maxima larger than 
$\frac{1}{p-1}$ are of interest)  yields the condition 
\begin{align*} 
{\rs}^{p-1} - \theta &= \lambda (p-1) ({\rs}^{p-1} - \theta {\rs}^{p-2}),  \\
\rs &\geq 0, \qquad  \lambda (p-1) \geq 1,  \qquad \theta \in \{-1,1\}. 
\end{align*}
We investigate the cases $\theta = 1$ and $\theta = -1$ separately. 

In case that $\theta = 1$, the  optimality condition read  
\[ 1 - \frac{1}{(\rs)^{p-1}} = \lambda (p-1)   (1-  \frac{1}{\rs}). \]
Thus, for $\lambda (p-1) >1$ to hold the following inequalities must be satisfied:
\[ 
\begin{cases}      
 1 - \frac{1}{(\rs)^{p-1}}  >  1-  \frac{1}{\rs} \Leftrightarrow 
 \rs < (\rs)^{p-1} 
\qquad & \mbox{ if } \rs >1, \\ 
 \frac{1}{(\rs)^{p-1}} -1 >  \frac{1}{\rs}-1 \Leftrightarrow 
 \rs > (\rs)^{p-1}  & \mbox{ if }  \rs <1.
\end{cases} 
\]
However in both cases the corresponding inequality cannot be true  since $p-1 \in (0,1)$. Thus in this case 
no interior maximum with a value larger than $\frac{1}{p-1}$  can exist.

In case that $\theta =-1$ the optimality condition reads 
\[ {\rs}^{p-1} +1 =  \lambda (p-1) ({\rs}^{p-1} +  {\rs}^{p-2}) 
 \geq ({\rs}^{p-1} +  {\rs}^{p-2}),  \]
thus, $\rs \geq 1$. 
Hence we have that 
\begin{equation} \frac{\f(\rs)}{\g(\rs)}= \lambda \leq \frac{1}{p-1} \max_{r\geq1} 
\frac{{r}^{p-1} + 1}{{r}^{p-1} +  {r}^{p-2}}  = 
\frac{1}{p-1} \left(1 +  \max_{r\geq1} \frac{1-{r}^{p-2} }{{r}^{p-1} +  {r}^{p-2}} \right).
\end{equation}
Since for $r \geq 1$ 
\[ \frac{{r}^{p-1} + 1}{{r}^{p-1} +  {r}^{p-2}} = 
\frac{{r}^{p-1} + 1}{{r}^{p-1}}  \frac{r}{r+1} \leq  2,
\]
we have  established the upper bound $\frac{2}{p-1}$.
\end{proof}

We now continue with the proof of Theorem~\ref{mainth}. 
We observe that for the $l^p$- or $L^p$-case we can express the Bregman distances componentwise:
\[ \Breg{\xi_y, \frac{1}{p}\|.\|_{l^p}^p}(x,y) = \sum_{i=1}^\infty  \Breg{\xi_{y_i}}(x_i,y_i), \] 
where  $\Breg{\xi_{y_i}}(x_i,y_i),$ is the Bregman distance for the functional $J:\R \to \R$, $J(x) = \frac{1}{p}|x|^p$. 
Thus, for \eqref{main1} to hold it is enough to prove the corresponding inequality for 
the Bregman distance of this functional, 
$\Breg{\xi_{y}}(x,y)$, $x,y \in \R$.

We have that  for $y \not = 0$ 
\[ \Breg{\xi_y}(x,y) = \|y\|^{p}\Breg{\xi_{e_y}}(\tfrac{x}{\|y\|},  e_y)  \]
and for $x \not = 0$, 
\[ \Breg{\xi_y}(x,y)  = \|x\|^{p}\Breg{\xi_z} (e_x,\tfrac{y}{\|x\|}), \qquad \xi_z \in \partial J(\tfrac{y}{\|x\|}). \]
Thus, for $y \not = 0$, \eqref{main1} is equivalent to 
\[ \Breg{\xi_{e_y}}(z,  e_y)  \leq C_M \Breg{\xi_{z}}( e_y,z) \]  
with $z = \tfrac{x}{\|y\|}$. 
We may calculate for $y \not = 0$ and   $\theta = 
 \dual{e_y, e_z} = \sign(y) \sign(z) $ and $z = \|z\| e_z$ that  
\[   \Breg{\xi_{e_y}}(z,  e_y) = f(\|z\|,\theta)   \qquad  \Breg{\xi_{z}}( e_y,z)  = g(\|z\|,\theta).  \] 
Thus, for the case $1 < p <2$, and $y \not = 0$, the theorem follows from 
Lemma~\ref{lemmatech}. The case of $y = 0$ can be estimated directly by 
$\Breg{\xi_0}(x,0) \leq \frac{1}{1-p} \Breg{\xi_x}(0,x)$. 

Considering now the case $2 < p < \infty$. From the duality formula \eqref{dualf} we obtain that 
 \eqref{main1} is equivalent to 
 \[ B_x^*(\xi_y,\xi_x) \leq C_M  B_y^*(\xi_x,\xi_y). \] 
Since $\partial J(\R) = \R$, and $J^* = \frac{1}{q}\|.\|^q$, $\frac{1}{q} + \frac{1}{p} = 1$ 
we obtain that \eqref{main1} holds with the constant 
\[ \frac{2}{q-1} = 2 (p-1), \] 
which finishes the proof. 

The analysis furthermore gives the more precise estimate 
\[ \frac{1}{t-1} \leq C_p \leq \frac{1}{t-1} \max_{r\geq1}
\frac{{r}^{t-1} + 1}{{r}^{t-1} +  {r}^{t-2}}  \qquad t = \begin{cases} p & 1 < p < 2  \\
        \frac{p}{p-1} & p \geq 2 \end{cases}. \]
The proof can be extended verbatim with the same constant to the case of $X$ begin the space of 
sequences with values in a Hilbert space and 
\[ J(x) = \frac{1}{p}\|x\|_{\ell^p(N,H)}^p =
\sum_{i=1}^n  \frac{1}{p} \|x_i\|_H^p   \]
 or Hilbert-space valued functions on an interval $L^p(I,H)$.
 
 \section*{Acknowledgement}
The work was  partly supported by the Austrian Science Fund (FWF) project
P 30157-N31. 

 \bibliographystyle{siam}
 \bibliography{mybib}
\end{document}